\documentclass[leqno,12pt]{amsart} 
\usepackage{amssymb}
\usepackage{amsmath}
\usepackage{amsthm}
\usepackage{amscd}
\usepackage{graphicx}
\usepackage[dvips]{color}
\usepackage[all]{xy}
\theoremstyle{plain} 
\newtheorem{theorem}{\indent\sc Theorem}[section]
\newtheorem{corollary}[theorem]{\indent\sc Corollary}
\newtheorem{proposition}[theorem]{\indent\sc Proposition}

\newtheorem{lemma}[theorem]{\indent\sc Lemma}
\theoremstyle{plain} 
\theoremstyle{definition} 

\newtheorem{remark}[theorem]{\indent\sc Remark}
\newtheorem{example}[theorem]{\indent\sc Example}

\makeatletter
    
    \@addtoreset{equation}{section}
  \makeatother

\begin{document}

\title[On a neighborhood of a torus leaf]
{On a neighborhood of a torus leaf of a certain class of holomorphic foliations on complex surfaces} 

\author[T. Koike]{Takayuki Koike} 

%
\address{
Graduate School of Mathematical Sciences, The University of Tokyo \endgraf
3-8-1 Komaba, Meguro-ku, Tokyo, 153-8914 \endgraf
Japan
}
\email{tkoike@ms.u-tokyo.ac.jp}

\maketitle

\begin{abstract}
Let $C$ be a smooth elliptic curve embedded in a smooth complex surface $X$ such that $C$ is a leaf of a suitable holomorphic foliation of $X$. 
We investigate complex analytic properties of a neighborhood of $C$ under some assumptions on complex dynamical properties of the holonomy function. 
As an application, we give an example of $(C, X)$ in which the line bundle $[C]$ is formally flat along $C$ however it does not admit a $C^\infty$ Hermitian metric with semi-positive curvature. 
\end{abstract}

\section{Introduction}
Let $X$ be a smooth complex surface and $C$ be a smooth elliptic curve embedded in $X$. 
Our aim is to investigate complex analytic properties of a neighborhood of $C$ when there exists a non-singular holomorphic foliation $\mathcal{F}$ on a neighborhood of $C$ of $X$ such that $C$ is a leaf of $\mathcal{F}$. 
Because of a technical reason, we always assume the following two conditions: 
$(1)$ there exists a neighborhood $W$ of $C$ and a holomorphic submersion $\pi\colon W\to C$ such that $\pi|_C$ is the identity map, and 
$(2)$ there exists a generator $\gamma_1$ of the fundamental group $\pi_1(C, *)$ of $C$ such that the holonomy of $\mathcal{F}$ on $C$ along $\gamma_1$ is trivial: 
i.e. the holonomy morphism ${\rm Hol}_C\colon\pi_1(C, *)\to \mathcal{O}_{\mathbb{C}, 0}$ satisfies 
\[
{\rm Hol}_C(\gamma_1)(\xi)=\xi,\ {\rm Hol}_C(\gamma_2)(\xi)=f(\xi)
\]
for (a germ of) some holomorphic function $f\in\mathcal{O}_{\mathbb{C}, 0}$ with $f(0)=0$ and $f'(0)\not=0$, where $\gamma_1$ and $\gamma_2$ are the generators of $\pi_1(C, *)$. 

The main result in the present paper is the following: 

\begin{theorem}\label{thm:main}
Let $X$ be a smooth complex surface and $C$ be a smooth elliptic curve embedded in $X$. 
Assume that there exists a non-singular holomorphic foliation $\mathcal{F}$ on a neighborhood of $C$ of $X$ such that $C$ is a leaf of $\mathcal{F}$ and the conditions $(1)$ and $(2)$ above hold for some $f\in\mathcal{O}_{\mathbb{C}, 0}$ with $f(0)=0$ and $f'(0)\not=0$. 
Then the following holds: \\
$(i)$ Assume that $f$ is a rational function and that $0$ is a repelling fixed point (i.e. $|f'(0)|>1$), an attracting fixed point (i.e. $|f'(0)|<1$), or a Siegel fixed point of $f$ (i.e. $f'(0)=e^{2\mathrm{i}\pi\theta}$ for some irrational number $\theta$ and $0$ lies in the Fatou set of $f$). 
Then there exits a neighborhood $W'$ of $C$ and a harmonic function $\Phi$ defined on $W'\setminus C$ such that $\Phi(p)=-\log{\rm dist}\,(p, C)+O(1)$ as $p\to C$, where ${\rm dist}\,(p, C)$ is a local Euclidean distance from $p$ to $C$. Especially, $C$ admits a pseudoflat neighborhood system. \\
$(ii)$ Assume that $0$ is a rationally indifferent fixed point of $f$ (i.e. $f'(0)=e^{2\mathrm{i}\pi\theta}$ for some rational number $\theta$) and that $n$-th iterate $f^n$ of $f$ is not equal to the identity map around $0$ for each integer $n$. 
Then $C$ admits a strongly pseudoconcave neighborhood system. \\
$(iii)$ Assume that $f$ is a polynomial and that, for each neighborhood $\Omega$ of $0$, there exists a periodic cycle $\{f(\eta), f^2(\eta), \dots, f^m(\eta)=\eta\}$ of $f$ included in $\Omega\setminus\{0\}$. 
Then, for any neighborhood $W'$ of $C$ and any continuous function $\psi\colon W'\to(-\infty, \infty]$ whose restriction $\psi|_{W'\setminus C}$ is psh (plurisubharmonic), 
$\psi$ is bounded from above on a neighborhood of $C$. 
\end{theorem}

Note that there actually exist examples of $(C, X)$ satisfying the assumptions in Theorem \ref{thm:main} for each of the statements $(i), (ii)$, and $(iii)$. 
It is because, as we will see in \S \ref{subsection:constr}, 
we can construct a pair $(C, X)$ which satisfies the conditions $(1)$ and $(2)$ 
for any fixed elliptic curve $C$ and any holnomy function $f\in\mathcal{O}_{\mathbb{C}, 0}$ 
(here we also need the facts on the existence of $f$ with $0$ as a Siegel fixed point \cite{Si} \cite[Theorem 6.6.4]{Be}, or a fixed point as in $(iii)$ \cite[\S 5.4]{U83}). 
Theorem \ref{thm:main} $(i)$ can be shown by concrete construction of such a harmonic function $\Phi$, and Theorem \ref{thm:main} $(ii)$ is obtained as a simple application of Ueda theory \cite{U83} (note that the normal bundle $N_{C/X}$ is topologically trivial, which follows directly from Camacho--Sad formula \cite{CS} \cite{S}). 
Thus, our main interest here is the case where $0$ is a Cremer fixed point of $f$  (i.e. 
$0$ is a irrationally indifferent fixed point lying in the Julia set of $f$), of which the situation of Theorem \ref{thm:main} $(iii)$ is a special case. 
We will show Theorem \ref{thm:main} $(iii)$ by the same technique as that used in the proof of \cite[Theorem 2]{U83} (by constructing leafwise harmonic psh function on a neighborhood of $C$ instead of the function whose complex Hessian has a negative eigenvalue). 

Here let us explain our motivation. 
Our original interest is the singularity of {\it minimal singular metrics} on 
a topologically trivial line bundle on a surface which is defined by a smooth embedded curve. 
Minimal singular metrics of a line bundle $L$ are metrics
of $L$ with the mildest singularities among singular Hermitian metrics of $L$ whose local
weights are psh. Minimal singular metrics were introduced in \cite[1.4]{DPS00}
as a (weak) analytic analogue of the Zariski decomposition. 
Let $X$ be a surface and $C$ be a smooth embedded curve with topologically trivial normal bundle, and denote by $[C]$ the line bundle defined by the divisor $C$. 
Ueda classified such a pair $(C, X)$ into the tree types: 
$(\alpha)$ when $[C]$ is not formally flat along $C$, $(\beta)$ when $[C]$ is flat around $C$, and $(\gamma)$ when $[C]$ is formally flat along $C$ however it is not flat around $C$. 
In \cite{K3}, we determined a minimal singular metric of $[C]$ when the pair $(C, X)$ is of type $(\alpha)$. 
From the argument in the proof of \cite[Corollary 3.4]{K2}, it can be shown that $[C]$ is semi-positive (i.e. $[C]$ admits a $C^\infty$ Hermitian metric with semi-positive curvature) when the pair $(C, X)$ is of type $(\beta)$. 
Then now we are interested in the case of type $(\gamma)$, especially for the example of $(C, X)$ of type $(\gamma)$ constructed in \cite[\S 5.4]{U83}, that is a motivation of this paper 
(here we note that, the setting of $(C, X)$ $(iii)$ in Theorem \ref{thm:main} is a modest generalization of this Ueda's example). 

As an application of Theorem \ref{thm:main}, we show the following: 

\begin{corollary}\label{cor:main}
Let $(C, X, \mathcal{F}, f)$ be that in Theorem \ref{thm:main}. 
Assume that $f$ is a polynomial of degree $d$ with $0$ as an irrationally indifferent fixed point. 
Denote by $\tau$ the number $f'(0)$. 
Then the following holds: \\
$(i)$ 
If there exists a positive number $M$ and $k$ such that $|\tau^n-1|^{-1}\leq M\cdot n^k$ holds for each integer $n$, 
then $[C]$ is semi-positive. \\
$(ii)$ If there is a number $A >1$ such that $\liminf_{\ell\to\infty}A^\ell\cdot|1-\tau^{\ell}|^{\frac{1}{d^\ell-1}}= 0$, then the singular Hermitian metric $|f_C|^{-2}$ is a metric on $[C]$ with the mildest singularities among singular Hermitian metrics $h$ on $[C]$ with semi-positive curvature such that $|f_C|_h$ is continuous around $C$, where $f_C\in H^0(X, [C])$ is a section whose zero divisor is $C$. 
Especially, there exists a pair $(C, X)$ of type $(\gamma)$ such that $[C]$ is not semi-positive. 
\end{corollary}

As an application of Corollary \ref{cor:main}, we construct a family 
of pairs of a surface 
and a line bundle 
defined on the surface whose semi-positivity 
varies pathologically depending on 
the parameter (see Example \ref{ex:cor}). 

The organization of the paper is as follows. 
In \S 2, we prove the existence and uniqueness (up to shrinking $X$) of the pair $(C, X)$ for a fixed holonomy function $f\in\mathcal{O}_{\mathbb{C}, 0}$. 
In \S 3, we prove Theorem \ref{thm:main} and Corollary \ref{cor:main}. 
In \S 4, We give some examples. 
\vskip3mm
{\bf Acknowledgment. } 
The author would like to give heartful thanks to Prof. Shigeharu Takayama and Prof. Tetsuo Ueda whose comments and suggestions were of inestimable value. 
He also thanks Prof. Masanori Adachi and Dr. Ryosuke Nomura for helpful comments and warm encouragements. 
He is supported by the Grant-in-Aid for Scientific Research (KAKENHI No.25-2869). 


\section{Construction and uniqueness of $(C, X)$}

\subsection{construction of $(C, X)$}\label{subsection:constr}

Let $f$ be a holomorphic function defined on a neighborhood $\Omega$ of $0$ in $\mathbb{C}$ such that $f(0)=0$ and $f'(0)\not=0$, and $C=\mathbb{C}^*/\sim_\lambda$ be a smooth elliptic curve, where $0<\lambda<1$ is a constant and $\sim_\lambda$ is the relation on $\mathbb{C}^*:=\mathbb{C}\setminus\{0\}$ generated by $z\sim_\lambda \lambda\cdot z$. 
We denote by $p$ the natural map $\mathbb{C}^*\to C$. 
In this subsection, we construct a smooth complex surface $X$ as in Theorem \ref{thm:main}: 
i.e. the surface $X$ which includes $C$ as a submanifold, admits a non-singular holomorphic foliation $\mathcal{F}$ on a neighborhood of $C$ of $X$ such that $C$ is a leaf of $\mathcal{F}$, and satisfies the following two conditions $(1)$ and $(2)$: 
$(1)$ there exists a neighborhood $W$ of $C$ and a holomorphic submersion $\pi\colon W\to C$ such that $\pi|_C$ is the identity map, and 
$(2)$ the holonomy morphism ${\rm Hol}_C\colon\pi_1(C, *)\to \mathcal{O}_{\mathbb{C}, 0}$ satisfies 
\[
{\rm Hol}_C(\gamma_1)(\xi)=\xi,\ {\rm Hol}_C(\gamma_2)(\xi)=f(\xi), 
\]
where $\gamma_1:=[p(\{|z|=1\})]\in\pi_1(C, *)$ and $\gamma_2:=[p(\{z\in\mathbb{R}\mid \lambda\leq z\leq 1\})]\in\pi_1(C, *)$. 

First, fix a sufficiently small neighborhood $U_0$ of $0$ in $\Omega$ such that $f|_{U_0}$ is injective. 
Denoting by $V_0$ the image $f(U_0)\subset \mathbb{C}$, let us consider the sets 
$V_1:=U_0\cap V_0$ and $U_1:=(f|_{U_0})^{-1}(V_1)$. 
In what follows, we regard $f$ as an isomorphism from $U_1$ to $V_1$. 
Fixing a sufficiently small positive constant $\varepsilon_0$, define the constants $\lambda_1$ and $\lambda_2$ by $\lambda_1:=1-\varepsilon_0, \lambda_2:=1+\varepsilon_0$. 
Denote by $X_1$ the set $p(\{z\in\mathbb{C}\mid \lambda<|z|<1\})\times U_0$, and by $X_2$ the set $p(\{z\in\mathbb{C}\mid \lambda_1<|z|<\lambda_2\})\times U_1$. 

Next, we construct $X$ by gluing $X_1$ and $X_2$ as follows. 
Let us denote by $X^-_1$ the subset 
$p(\{z\in\mathbb{C}\mid \lambda_1<|z|<1\})\times U_1$ of $X_1$, 
and by $X^-_2$ the subset $p(\{z\in\mathbb{C}\mid \lambda_1<|z|<1\})\times U_1$ of $X_2$. 
We glue them up by the isomorphism $i^-\colon X^-_2\to X^-_1$ defined by $i^-(p(z), \xi):=(p(z), \xi)$. 
Denote by $X^+_1$ the subset 
$p(\{z\in\mathbb{C}\mid 1<|z|<\lambda_2\})\times V_1$ of $X_1$, and 
by $X^+_2$ the subset $p(\{z\in\mathbb{C}\mid 1<|z|<\lambda_2\})\times U_1$ of $X_2$. 
We glue them up by the isomorphism $i^+\colon X^+_2\to X^+_1$ defied by $i^+(p(z), \xi):=(p(z), f(\xi))$. 
Then $X_1$ and $X_2$ glue up to each other by the morphisms $i^+$ and $i^-$ above to define a smooth complex surface, by which we define $X$. 

Finally, we will check that this $X$ satisfies the conditions above. 
Note that the first projections $X_1\to p(\{z\in\mathbb{C}\mid \lambda<|z|<1\})$ and $X_2\to p(\{z\in\mathbb{C}\mid \lambda_1<|z|<\lambda_2\})$ glue up to each other to define a entire map $\pi\colon X\to C$. 
As this morphism $\pi$ is a holomorphic submersion, we can check the condition $(1)$ (the condition on $\pi|_C$ can be easily checked by the following construction of the submanifold $C\subset X$. Note also that, in this construction, $W'$ can be taken as $X$ itself). 
Next, we will define a foliation $\mathcal{F}$ on $X$. 
Let $\mathcal{F}_1$ be the foliation on $X_1$ whose leaves are $\{\{(p(z), \xi)\in X_1\mid \xi=c\}\}_{c\in U_0}$, 
and $\mathcal{F}_2$ be the foliation on $X_2$ whose leaves are $\{\{(p(z), \xi)\in X_2\mid \xi=c\}\}_{c\in U_1}$. 
These two foliations glue up to each other by the morphisms $i^+$ and $i^-$ above to define the foliation on $X$, which we denote by $\mathcal{F}$. 
As $f(0)=0$, the leaves $\{(p(z), \xi)\in X_1\mid \xi=0\}$ and $\{(p(z), \xi)\in X_2\mid \xi=0\}$ glue up to define a compact connected leaf of $\mathcal{F}$, which is naturally isomorphic to $C$. We regard this compact leaf as a submanifold $C\subset X$. 
From this construction, one can easily check the condition $(2)$. 

\subsection{uniqueness of $(C, X)$}

Here we will show the following: 

\begin{proposition}\label{prop:uniq}
Let $f$ be an element of $\mathcal{O}_{\mathbb{C}, 0}$ such that $f(0)=0$ and $f'(0)\not=0$ and $C$ be a smooth elliptic curve. 
Let $X'$ be a surface as in Theorem \ref{thm:main}: 
i.e. $X'$ includes $C$ as a submanifold, admits a non-singular holomorphic foliation $\mathcal{F}'$ on a neighborhood of $C$ of $X'$ such that $C$ is a leaf of $\mathcal{F}'$, and satisfies the following two conditions $(1)$ and $(2)$: 
$(1)$ there exists a neighborhood $W'$ of $C$ and a holomorphic submersion $\pi'\colon W'\to C$ such that $\pi'|_C$ is the identity map, and 
$(2)$ the holonomy morphism ${\rm Hol}'_C\colon\pi_1(C, *)\to \mathcal{O}_{\mathbb{C}, 0}$ satisfies 
\[
{\rm Hol}'_C(\gamma_1)(\xi)=\xi,\ {\rm Hol}'_C(\gamma_2)(\xi)=f(\xi), 
\]
where $\gamma_1$ and $\gamma_2$ are those in \S \ref{subsection:constr}. 
Then, by shrinking $U_0$ in \S \ref{subsection:constr} if necessary, there exists an holomorphic map $j\colon X\to W'$ such that $j$ is an isomorphism to the image of $j$, $j(C)=C\subset X'$, $j$ preserves the foliation structures, and that 
$\pi'\circ j=\pi$ holds, where $(X, \pi, \mathcal{F})$ are those in \S \ref{subsection:constr}. 
\end{proposition}

\begin{proof}
By shrinking $X'$, we may assume that $W'=X'$. 
Denote by $p$ the point $p((1+\lambda_2)/2)\in C$, and fix an embedding $\pi'^{-1}(p)\to\mathbb{C}$ (by shrinking $W'$ if necessary). 
We also may assume that $U_0$ is small enough so that we can regard it as a subset of $\pi'^{-1}(p)$: $U_0\subset \pi'^{-1}(p)$. 
Let $h_1\colon \pi'^{-1}(p(\{\lambda<|z|<1\}))\to \mathbb{C}$ be the leafwise constant holomorphic extension of the inclusion $\pi'^{-1}(p)\to\mathbb{C}$, 
and $h_2\colon \pi'^{-1}(p(\{\lambda_1<|z|<\lambda_2\}))\to \mathbb{C}$ be the leafwise constant holomorphic extension of  the inclusion $\pi'^{-1}(p)\to\mathbb{C}$ (note that the condition on the holonomy along $\gamma_1$ is needed for the existence of $h_1$ and $h_2$). 
Letting $X_1':=\{q\in \pi'^{-1}(p(\{\lambda<|z|<1\}))\mid h_1(q)\in U_0\}$ and 
$X_2':=\{q\in \pi'^{-1}(p(\{\lambda_1<|z|<\lambda_2\}))\mid h_2(q)\in V_1\}$, consider the maps
\[
X_1'\to X_1\colon q\to (\pi'(q), h_1(q)),\ X_2'\to X_2\colon q\to (\pi'(q), (f|_{U_1})^{-1}(h_2(q))). 
\]
As these maps are holomorphic and bijective, and as $\pi'$ is submersion, we can conclude that both of these two maps are isomorphisms. 
By using the condition on the holonomy along $\gamma_2$, we can easily show that the subset $X_1'\cup X_2'$ of $X'$ is isomorphic to $X$ and the proposition holds. 
\end{proof}

\section{Proof}

\subsection*{Proof of Theorem \ref{thm:main}}

Let $f\in\mathcal{O}_{\mathbb{C}, 0}$ be an element such that $f(0)=0$ and $\tau:=f'(0)\not=0$. 
From Proposition \ref{prop:uniq}, It is sufficient to show Theorem \ref{thm:main} for $(C, X, \mathcal{F}, \pi)$ we constructed in \S \ref{subsection:constr}. 

\subsubsection*{Proof of Theorem \ref{thm:main} $(i)$}

Assume that $0$ is a repelling fixed point, an attracting fixed point, or a Siegel fixed point of $f$. 
According to \cite[Theorem 6.3.2, 6.6.2]{Be} and the comment near the proof of \cite[Theorem 6.4.1]{Be}, there exists an element $h\in\mathcal{O}_{\mathbb{C}, 0}$ such that $h(0)=0$, $h'(0)=1$, and $f(h(\xi))= h(\tau\cdot\xi)$ holds. 
Thus, without loss of generality, we may assume that $f$ is linear: $f(\xi)=\tau\cdot \xi$. 

First we define a function $\Phi$ on $X$ as follows: 
Define the function $\Phi_1$ on $X_1$ by 
\[
\Phi_1(p(z), \xi) := \log|\xi|+a\cdot \log|z|, 
\]
where $a:=\frac{-\log |\tau|}{\log\lambda}$ and  $z$ is a complex number such that $\lambda<|z|<1$ and $\xi\in U_0$. 
Also define the function $\Phi_2$ on $X_2$ by
\[
\Phi_2(p(z), \xi) := \log|\xi|+a\cdot \log|z|
\]
for each $z$ such that $\lambda_1<|z|<\lambda_2$ and $\xi\in U_1$. 
It is clear that $(i^-)^*\Phi_1=\Phi_2|_{X^-_2}$ holds.  
The equality $(i^+)^*\Phi_1=\Phi_2|_{X^+_2}$ also can be shown by the calculation as follows: for each $z$ such that $1<|z|<\lambda_2$, 
\[
(i^+)^*\Phi_1(p(z), \xi)=\Phi_1(p(z), f(\xi))=
\log|f(\xi)|+a\log|\lambda\cdot z|=
\Phi_2(p(z), \xi). 
\]
Now we showed that the functions $\Phi_1$ and $\Phi_2$ glue up to define a function, by which we define the function $\Phi$. 

Clealy $\Phi$ is harmonic with $\Phi(p)=-\log{\rm dist}\,(p, C)+O(1)$ as $p\to C$, which shows Theorem \ref{thm:main} $(i)$. 

\subsubsection*{Proof of Theorem \ref{thm:main} $(ii)$}

Assume that $f$ is non-linear and that $0$ is a rationally indifferent fixed point of $f$. 
Then we may assume that $f$ has the expansion $f(\xi)=\tau\cdot\xi+A\cdot \xi^{n+1}+O(\xi^{n+2})$ for some integer $n$ $(A\not=0)$. 
Let us consider the cohomology class
\[
u_n:=[\{(X^-_2\cap C, 0), (X^+_2\cap C, \tau^{-1}\cdot A)\}]\in H^1(C, N_{C/X}^{-n}). 
\]
If $u_n=0$ holds, then there exists a $0$-cochain $\{(X_1\cap C, h_1), (X_2\cap C, h_2)\}\in C^0(C, N_{C/X}^{-n})$ such that 
$\delta\{(X_1\cap C, h_1), (X_2\cap C, h_2)\}=\{(X^-_2\cap C, 0), (X^+_2\cap C, \tau^{-1}\cdot A)\}$: i.e. 
\[
h_1=
\begin{cases}
    h_2 & ({\rm on}\ X^-_2\cap C) \\
    \tau^{n}h_2-\tau^{-1}\cdot A & ({\rm on}\ X^+_2\cap C)
  \end{cases}
\]
holds. 
Without loss of generality,  we may assume that both $h_1$ and $h_2$ are constant functions (It is because, if $\tau^n\not=1$, then we can use the constant functions $h_1=h_2=(\tau^n-1)^{-1}\cdot \tau^{-1}\cdot A$. 
If $\tau^n=1$, then the holomorphic functions $\exp(2\textrm{i}\pi\tau A^{-1}\cdot h_1)$ and $\exp(2\textrm{i}\pi\tau A^{-1}\cdot h_2)$ glue up to define a entire non-vanishing holomorphic function on $C$, which shows that $h_1$ and $h_2$ are constants). 
Note that one can deduce directly from the above that the constant $h_1$ satisfies $h_1= \tau^{n}h_1-\tau^{-1}\cdot A$. Let us set 
\[
h(\xi):=\xi+h_1\cdot\xi^{n+1}. 
\]
Then 
\begin{eqnarray}
h^{-1}f(h(\xi))&=&h^{-1}(\tau\cdot(\xi+h_1\cdot\xi^{n+1})+ A\cdot\xi^{n+1}+O(\xi^{n+2}))\nonumber \\
&=& h^{-1}(\tau\cdot\xi+(\tau\cdot h_1+A)\cdot\xi^{n+1}+O(\xi^{n+2}))\nonumber \\
&=& (\tau\cdot\xi+(\tau\cdot h_1+A)\cdot\xi^{n+1})-h_1\cdot(\tau\cdot\xi)^{n+1}+O(\xi^{n+2})\nonumber \\
&=& \tau\cdot\xi+\tau\cdot((1-\tau^{n})\cdot h_1+\tau^{-1}A)\cdot\xi^{n+1}+O(\xi^{n+2})\nonumber \\
&=& \tau\cdot\xi+O(\xi^{n+2})\nonumber
\end{eqnarray}
holds. 
Thus, by using new coordinate $\xi ':=h(\xi)$ instead of $\xi$, 
we can expand $f$ as $f(\xi')=\tau\cdot\xi'+O(\xi'^{n+2})$, and thus we can define the cohomology class $u_{n+1}$ just as in the same manner. 
Note that these $u_n$ defined as above is equal to the trivial element $0\in H^1(C, N_{C/X}^{-n})$ if and only if the obstruction class $u_n(C, X)$ is trivial 
(it is clear from the definition of the obstruction class $u_n(C, X)$, see \cite[\S 2.1]{U83}, see also \cite[\S 3]{K3}). 


Assume that $u_n\not=0$ holds for some integer $n$. 
In this case, as the pair $(C, X)$ is of type $(\alpha)$ in the classification by Ueda \cite[\S 5]{U83}, we can apply \cite[Theorem 1]{U83} and its corollary, which proves Theorem \ref{thm:main} $(ii)$. 

Therefore all we have to do is to show that there actually exists an integer $n$ such that $u_n\not=0$ holds. 
Assume that $u_n=0$ for all $n$. 
Let $m\geq 1$ be an integer such that $\tau^m=1$. 
Then, from the assumption, $f^m$ can be expanded as below: 
\[
f^m(\xi)=\xi+A\cdot \xi^{n+1}+O(\xi^{n+2}), 
\]
where $A$ in a non-zero constant. 
Since $u_1=u_2=\dots, u_n=0$, we can choose a suitable polynomial $h(\xi)=\xi+O(\xi^2)$ such that $h^{-1}\circ f\circ h(\xi)=\tau\cdot\xi+O(\xi^{n+2})$. 
Let 
\[
h^{-1}(\eta)=\eta+\sum_{\nu=2}^\infty b_\nu\cdot\eta^\nu
\]
be the expansion of the inverse function $h^{-1}$ of $h$ around $0$. 
Then we can calculate that 
\begin{eqnarray}
h^{-1}\circ f^m\circ h(\xi)
&=& h^{-1}(h(\xi)+A\cdot (h(\xi))^{n+1}+O(\xi^{n+2}))\nonumber \\
&=& h^{-1}(h(\xi)+A\cdot \xi^{n+1}+O(\xi^{n+2}))\nonumber \\
&=& (h(\xi)+A\cdot \xi^{n+1}+O(\xi^{n+2}))+\sum_{\nu=2}^\infty b_\nu\cdot (h(\xi)+A\cdot \xi^{n+1}+O(\xi^{n+2}))^\nu\nonumber 
\end{eqnarray}
\begin{eqnarray}
&=& (h(\xi)+A\cdot \xi^{n+1}+O(\xi^{n+2}))+\sum_{\nu=2}^\infty b_\nu\cdot (h(\xi)^\nu+O(\xi^{n+2}))\nonumber \\
&=& h^{-1}(h(\xi))+A\cdot \xi^{n+1}+O(\xi^{n+2})=\xi+A\cdot \xi^{n+1}+O(\xi^{n+2}), \nonumber 
\end{eqnarray}
and also that 
\begin{eqnarray}
h^{-1}\circ f^m\circ h(\xi)
&=& (h^{-1}\circ f\circ h)^m(\xi)\nonumber \\
&=& (h^{-1}\circ f\circ h)^{m-1}(\xi+O(\xi^{n+2}))
= \dots
= \xi+O(\xi^{n+2}), \nonumber 
\end{eqnarray}
which leads the contradiction. 

\subsubsection*{Proof of Theorem \ref{thm:main} $(iii)$}

Assume that $f$ is a polynomial of degree $d$ and that, for each neighborhood $\Omega$ of $0$, there exists a periodic cycle of $f$ included in $\Omega\setminus\{0\}$. 
Let $g$ be the Green function of the filled Julia set $K(f)$. 
Note that $g|_{K(f)}\equiv 0$, $f^*g=d\cdot g$, and that $g|_{I(f)}$ is a harmonic function valued in $\mathbb{R}_{>0}$, where $I(f):=\mathbb{C}\setminus K(f)=\{\xi\in\mathbb{C}\mid f^n(\xi)\to\infty\ {\rm as}\ n\to \infty\}$. 
Note also that $g$ is H\"older continuous (see
\cite[\S VIII, Theorem 3.2]{CG}). 

\begin{lemma}\label{lem:cremer}
The point $0$ lies in the boundary of the set $\{\xi\in\mathbb{C}\mid g(\xi)=0\}$. 
\end{lemma}

\begin{proof}
As the total number of non-repelling cycles of $f$ is finite (see \cite[\S III Theorem 2.7]{CG}) and every repelling cycle of $f$ lies in $J(f)$ (see \cite[Theorem 6.4.1]{Be}), 
we can conclude from the assumption that $0\in J(f)$. 
Thus the lemma follows from the fact that $J(f)$ coincides with the boundary $\partial K(f)$ of $K(f)$ (see \cite[\S III.4]{CG}). 
\end{proof}

Let $W'$ be a neighborhood of $C$ and $\psi\colon W'\to (-\infty, \infty]$ be a continuous function whose restriction $\psi|_{W'\setminus C}$ is psh. 
By shrinking $W'$, we may assume that $\psi$ is bounded from above on a neighborhood of $\partial W'$. 
Assuming that $\psi$ is not bounded from above, we will derive a contradiction. 
By shrinking $U_0$ in \S \ref{subsection:constr} if necessary, we will assume that $W'=X$. 

First we construct the function $G$ on $X$ as follows: 
Setting $a:=\frac{-\log d}{\log\lambda}$, 
define the function $G_1$ on $X_1$ by 
\[
G_1(p(z), \xi) := g(\xi)\cdot |z|^a 
\]
for $z$ such that $\lambda<|z|<1$ and $\xi\in U_0$. 
Similarly, define the function $G_2$ on $X_2$ by 
\[
G_2(p(z), \xi) := g(\xi)\cdot |z|^a
\]
for $z$ such that $\lambda_1<|z|<\lambda_2$ and $\xi\in U_1$. 
Note that $(i^-)^*G_1=G_2|_{X^-_2}$. 
Note also that $(i^+)^*G_1=G_2|_{X^+_2}$ holds, since 
\[
(i^+)^*G_1(p(z), \xi)=G_1(p(z), f(\xi))=g(f(\xi))\cdot |\lambda\cdot z|^a=d\lambda^a\cdot g(\xi) \cdot|z|^a=g(\xi) \cdot|z|^a
\]
holds for each for each $z$ such that $1<|z|<\lambda_2$. 
Thus $G_1$ and $G_2$ glue up to define a function on $X$, by which we define $G$. 

By shrinking $U_0$ if necessary, we may assume that $G< 1$ holds on $X$. 
We also assume that $\psi<0$ holds on a neighborhood of the boundary $\partial X$ of $X$ by replacing $\psi$ with $\psi-M$ for sufficiently large real number $M$ if necessary. 
Then the following lemma holds: 

\begin{lemma}\label{lemma:key}
There exists a connected leaf $L$ of $\mathcal{F}$ such that $\overline{L}\cap \partial X\not=\emptyset$
and there exists a interior point $p\in L$
such that $p$ attains the maximum value $B$ of the function $H:=\frac{\psi}{-\log G}$ on $L$. 
Moreover, $B$ is a positive real number. 
\qed
\end{lemma}

\begin{proof}
Consider the function
\[
H^*(q):=\limsup_{\zeta\to q}H(\zeta), 
\]
which is a upper semi-continuous extension of $H$. 
Then, as the function $\psi$ is locally bounded from above on $X\setminus C$, it is clear that $H^*(q)=0$ holds for each point $q\in\{G=0\}\setminus C$. 
From the assumption, there exists a point $p_0\in C$ such that $\psi(p_0)=\infty$ holds. 
Fix a sufficiently small neighborhood $U_2$ of $p_0$ in $\pi^{-1}(\pi(p_0))$ such that $\psi|_{U_2}>0$ and regard it as a subset of $U_1(\subset U_0)$ (here the continuity assumption for $\psi$ is needed). 
From the assumption and the fact that the total number of non-repelling cycles of $f$ is finite (\cite[\S III Theorem 2.7]{CG}), 
there exists a repelling cycle $\{f(\eta), f^2(\eta), \dots, f^m(\eta)=\eta\}\subset  U_2\setminus \{0\}$. 
Fix a sequence $\{\eta_n\}_{n\geq 0}\subset U_1\setminus K(f)$ such that $f(\eta_{n+1})=\eta_n$ and the set of all accumulation points of $\{\eta_n\}_n$ coincides with the cycle $\{f^n(\eta)\}_n$ (see Lemma \ref{lem:leaf} for the existence of such a sequence). 
Clearly there is a connected leaf $L$ of $\mathcal{F}$ such that, for sufficiently large $n$, the point $\eta_n\in U_2\subset\pi^{-1}(\pi(p_0))$ lies in $L$. 
Note that $\overline{L}\cap \partial X\not=\emptyset$ follows immediately from $\eta_0\in U_1\setminus K(f)\subset I(f)$. 
%

Since the function $\psi|_L$ is negative around $\overline{L}\cap \partial X$ and $H^*(f^n(\eta))=0$ holds for each $n$, the set $\{q\in L\mid H^*(q)>0\}$ is relatively compact subset of $L$ 
(note that, as $\psi|_{U_2}>0$ holds, the set $\{q\in L\mid H^*(q)>0\}$ is not empty). 
Thus it follows from the upper semi-continuity of the function $H^*|_L$ that there exists a point $p\in\{q\in L\mid H^*(q)>0\}$ which attains the maximum $B>0$ of the function $H^*|_L$. 
\end{proof}

\begin{lemma}\label{lem:leaf}
Let $\eta\in J(f)\cap U_1$ be a point included in a repelling cycle of $f$. 
Then there exists a sequence $\{\eta_n\}_{n\geq 0}\subset U_1\setminus K(f)$ such that $f(\eta_{n+1})=\eta_n$ and the set of all accumulation points of $\{\eta_n\}_n$ coincides with the cycle $\{f^n(\eta)\}_n$. 
\end{lemma}

\begin{proof}
Let $m$ be the minimum positive integer which satisfies $f^m(\eta)=\eta$. 
Then $\eta$ is a repelling fixed point of the polynomial $f^m\colon \mathbb{C}\to\mathbb{C}$. 
Thus, by choosing suitable coordinate $\xi'$ of $\mathbb{C}$ such that $\xi'(\eta)=0$, 
we may assume that $f(\xi')=\tau_\eta\cdot \xi'$ for some complex number $\tau_\eta$ such that $|\tau_\eta|>1$ 
(see the comment near the proof of \cite[Theorem 6.4.1]{Be}). 
Fix a point $\xi_0\in U_1\setminus K(f)$ which is sufficiently close to the point $\eta$ 
and define the points $\eta_{j\cdot m}$ by $\xi'(\eta_{j\cdot m})=\tau_\eta^{-j}\cdot\xi'(\xi_0)$ for each integer $j$. 
For each integers $j\geq1$ and $n\in\{0, 1, \dots, m-1\}$, set $\eta_{m\cdot j-n}:=f^n(\eta_{m\cdot j})$. 
Then one can easily check the equation $f(\eta_{n+1})=\eta_n$. 

Let $\xi_1$ be an accumulation point of the sequence $\{\eta_n\}$. 
Then there exits a subsequence $\{n_k\}_k\subset\mathbb{Z}_{\geq 0}$ such that  $\lim_{k\to\infty}\eta_{n_k}=\xi_1$holds. 
Fix an integer $\ell\in\{0, 1, \dots, m-1\}$ such that the sub sequence $\{k_j\}_j:=\{k\mid n_k\equiv -\ell\ {\rm mod}\ m\}$ is infinite. 
Letting $n_{k_j}=\nu_j\cdot m-\ell$, we can calculate 
\[
\xi_1
=\lim_{j\to\infty}\eta_{n_{k_j}}
=\lim_{j\to\infty}\eta_{\nu_j\cdot m-\ell}
=\lim_{j\to\infty}f^{\ell}\left(\eta_{\nu_j\cdot m}\right)
=f^{\ell}\left(\lim_{j\to\infty}\eta_{\nu_j\cdot m}\right)
=f^\ell(\eta), 
\]
which shows the lemma. 
\end{proof}

\begin{proof}[End of proof of Theorem \ref{thm:main} $(iii)$]
Let $L, p, B$ be those in Lemma \ref{lemma:key}. 
As $B$ is the maximum of the function $H=\frac{\psi}{-\log G}$ on $L$, the inequality 
$\psi+B\cdot\log G\leq 0$ holds on $L$. 
As the function $\psi$ is psh and the function $(\log G)|_L$ is harmonic, we can conclude that $(\psi+B\cdot\log G)|_L$ is a subharmonic function on $L$. 
Since $\psi(p)+B\cdot\log G(p)=0$ holds and $p$ is an interior point of $L$, one can use maximum principle to show that $(\psi+B\cdot\log G)|_L\equiv 0$ holds. 
Thus 
\[
\psi|_L\equiv(-B\cdot\log G)|_L
\]
holds, which leads the contradiction since $\psi|_L<0$ holds around $\overline{L}\cap\partial X$ and $-B\cdot\log G>0$ holds on every point of $X$. 
\end{proof}

\begin{remark}
Here we give another (simplified) proof of Theorem \ref{thm:main} $(iii)$, 
which was taught by Professor Tetsuo Ueda. 

Let $\psi$ be a psh function defined on $W'\setminus C$. 
Fix a neighborhood $Y$ of $C$ in $W'$ such that $\overline{Y}\subset W'$ and $Y\subset W$. 
Set $M := \sup_{\partial Y} \psi$ 
and fix a compact leaf $\Gamma$ of $\mathcal{F}$ such that $\Gamma\cap\pi^{-1}(p(1))\subset U_1$ is a repelling cycle of $f$. 
Let $L$ be a leaf of $\mathcal{F}$ which accumulates to $\Gamma$ (this $L$ can be constructed in the same manner as in the above proof of Theorem \ref{thm:main} $(iii)$). 
Fix a holomorphic map $g \colon \{ z\in\mathbb{C}\mid 0 < |z| \leq R\} \to L$ for some $R>0$ which is an isomorphism to the image such that $g(\{|z|=R\}) \subset Y$ and $g(z)\to \Gamma$ as $z\to 0$.  
As the function $\psi\circ g$  is a psh function defined on $\{z\in\mathbb{C}\mid 0<|z|<R\}$ bounded from above, we can extend it and can regard it as a psh function defined on $\{z\in\mathbb{C}\mid |z|<R\}$. 
Thus we can use maximum principle to conclude that $\psi(g(z)) \leq M$ holds for each $z$ such that $|z|<R$．
Therefor we obtain the inequality $\psi |_{L\cap Y} \leq M$. 
From the assumption, such a leaf $\Gamma$ exists in any neighborhood of $C$ in $X$. 
Thus we can conclude from the above inequality that $\liminf_{q\to p}\psi(q) \leq M$ holds for each pint $p \in C$, which shows Theorem \ref{thm:main} $(iii)$. 
\end{remark}

\subsection*{Proof of Corollary \ref{cor:main}}

\subsubsection*{Proof of Corollary \ref{cor:main} $(i)$}

In this case, $0$ is a Siegel fixed point of $f$ and thus there exists a function $\Phi$ as in the proof of Theorem \ref{thm:main} $(i)$. 
By using this function $\Phi$, we can conclude that there exists a flat metric on $[C]|_W$ for some neighborhood $W$ of $C$, which clearly has semi-positive curvature. 
By using this flat metric, we can construct a smooth Hermitian metric on $[C]$ with semi-positive curvature from the same arguments as in the proof of \cite[Corollary 3.4]{K2}. 

\subsubsection*{Proof of Corollary \ref{cor:main} $(ii)$}

In this case, there exists a periodic cycle of $f$ included in $\Omega\setminus\{0\}$ for each neighborhood $\Omega$ of $0$ 
(see \cite[\S 5.4]{U83}). 
Let $h$ be a singular Hermitian metric on $[C]$ with semi-positive curvature such that $|f_C|_h$ is continuous around $C$, 
where $f_C$ is a global holomorphic section of $[C]$ whose zero divisor coincides with the divisor $C$. 
Then, as the function $\psi:=-\log |f_C|_h^2$ can be regarded as a continuous function defined on a neighborhood of $C$ which is psh outside of $C$, we can conclude form Theorem \ref{thm:main} $(iii)$ that there exits a positive constant $M$ such that $\psi<M$ holds on $X$, which shows the corollary (see also the proof of \cite[Theorem 1.1]{K3}). 

\section{Some examples}

\begin{example}
Let $p\colon \mathbb{C}^*\to C$ and $0<\lambda<1$ be those in \S \ref{subsection:constr}. 
Consider the rank-2 vector bundle on $E\to C$ defined by $E:=(\mathbb{C}^*\times\mathbb{C}^2)/\sim$, where $\sim$ is the relation generated by $(z, x, y)\sim (\lambda\cdot z, x, x+y)$. 
Let $X$ be the ruled surface associated to $X$: $X:=\mathbb{P}(E)$. 
$X$ admits a non-singular holomorphic foliation $\mathcal{F}$ whose leaves are either 
$\{\{[(z, x, y)]\in X\mid y=(c+n)\cdot x\ \text{for some}\ n\in\mathbb{Z}\}\}_{c\in\mathbb{C}/\mathbb{Z}}$ or $\{[(z, x, y)]\in X\mid x=0\}$. 
As $\{[(z, x, y)]\in X\mid x=0\}$ is naturally isomorphic to $C$, let us regard it as $C$ embedded in $X$: $C\subset X$. 
Then $(X, C, \mathcal{F})$ enjoys the conditions in Theorem \ref{thm:main}. 
In this case, the holonomy map $f(\xi)$ can be calculated as $f(\xi)=\frac{\xi}{1+\xi}$, which has $0$ as a rationally indifferent fixed point. 

Note that $E$ is a rank-$2$ degree-$0$ vector bundle which is the non-trivial extension of $\mathbb{I}_C$ by $\mathbb{I}_C$, where $\mathbb{I}_C$ is the trivial line bundle on $C$. 
According to \cite[\S 6]{N}, this $X$ is essentially the unique example of projective smooth surface in which smooth elliptic curve can be embedded as a curve of type $(\alpha)$ in Ueda's classification. 
Note also that this example is the same one as \cite[Example 1.7]{DPS94} (see also \cite[\S 4.1]{E}). 
\end{example}

\begin{example}
Let $C_0$ be a smooth elliptic curve embedded in the projective plane $\mathbb{P}^2$. 
Fix nine points $p_1, p_2, \dots, p_9$ from $C_0$ different from each other. 
Let us denote by $X$ the blow-up of $\mathbb{P}^2$ at $\{p_j\}_{j=1}^9$, and by $C$ the strict transform of $C_0$. 
From the simple calculation, it follows that the degree of the normal bundle $N_{C/X}$ is equal to $0$, and thus $N_{C/X}$ is topologically trivial. 
By choosing $\{p_j\}_{j=1}^9$ in sufficiently general position, we may assume that $N_{C/X}$ is a non-torsion element of ${\rm Pic}^0(C)$: 
i.e. there does not exist an integer $\ell$ such that $N_{C/X}^\ell=\mathbb{I}_C$. 
We here remark that the neighborhood structure of $C$ and the semi-positivity of $[C]$ in this example has been deeply investigated \cite{Br} \cite{U83} (see also \cite[\S 2]{D}). 
In order to determine a minimal singular metric of $[C]$ in this example by using Corollary \ref{cor:main}, we are interested in whether there exists an configuration of nine points $\{p_j\}_{j=1}^9$ such that there exists a foliation $\mathcal{F}$ and a submersion $\pi$ which satisfies the conditions $(1)$ and $(2)$. 
Unfortunately, we cannot give any answer to this question 
(here we remark that the question on the existence of holomorphic foliation on this $X$ has already been posed in \cite{DPS96}). 
\end{example}

\begin{example}\label{ex:cor}
Let $C$ be a smooth elliptic curve. 
Here we construct an holomorphic submersion $\widetilde{X}\to\Omega$ from a smooth complex manifold $\widetilde{X}$ of dimension three to a neighborhood $\Omega$ of $U(1):=\{\tau\in\mathbb{C}\mid |\tau|=1\}$ in $\mathbb{C}$ which satisfies the following conditions: 
$(a)$ there exists a submanifold $\widetilde{C}$ of $\widetilde{X}$ of dimension two such that the restriction $\widetilde{C}\to \Omega$ is a proper submersion and each fiber of this restricted map is isomorphic to $C$, 
$(b)$ $[\widetilde{C}]|_{X_\tau}$ is semi-positive for each $\tau\in\Omega\setminus U(1)$, where $X_\tau$ is the fiber of $\tau$, 
$(c)$ $[\widetilde{C}]|_{X_\tau}$ is also semi-positive for almost all $\tau\in U(1)$ in the sense of Lebesgue measure, 
and $(d)$ there exist uncountably many elements $\tau\in U(1)$ such that $[\widetilde{C}]|_{X_\tau}$ is not semi-positive. 

Fix a sufficiently small open neighborhood $\Omega$ of 
$U(1)$ in $\mathbb{C}^*$ and 
consider the function $F\colon\mathbb{C}\times \Omega\to \mathbb{C}\times \Omega$ defined by $F(\xi, \tau):=(\tau\cdot\xi+\xi^2, \tau)$. 
Fix also a sufficiently small neighborhood $\widetilde{U}_0$ of 
$\{0\}\times U(1)$ in $\mathbb{C}\times \Omega$ such that $f|_{\widetilde{U}_0}$ is locally isomorphic (note that the Jacobian determinant of $F$ at $(0, \tau)$ is $\tau$, which is a non-zero constant for each $\tau\in \Omega$). 
We may assume that $\widetilde{U}_0=U_0\times \Omega$ holds for some neighborhood $U_0\subset \mathbb{C}$ of $0$. 
By shrinking $U_0$, we may assume that $F|_{\widetilde{U}_0}$ is injective and thus it is an isomorphism to the image of it. 
Denoting by $\widetilde{V}_0$ the image $F(\widetilde{U}_0)\subset \mathbb{C}\times \Omega$, let us consider the sets 
$\widetilde{V}_1:=\widetilde{U}_0\cap \widetilde{V}_0$ and $\widetilde{U}_1:=(F|_{\widetilde{U}_0})^{-1}(\widetilde{V}_1)$. 
In what follows, we regard $F$ as an isomorphism from $\widetilde{U}_1$ to $\widetilde{V}_1$. 
Let $p, \lambda, \lambda_1, \lambda_2$ be those in \S \ref{subsection:constr}. 
Denote by $\widetilde{X}_1$ the set $p(\{z\in\mathbb{C}\mid \lambda<|z|<1\})\times \widetilde{U}_0$, and by $\widetilde{X}_2$ the set $p(\{z\in\mathbb{C}\mid \lambda_1<|z|<\lambda_2\})\times \widetilde{U}_1$. 
We define a complex manifold $\widetilde{X}$ by gluing up $\widetilde{X}_1$ and $\widetilde{X}_2$ in the same manner as in \S \ref{subsection:constr} by using the function $F$. 
Denote by $\widetilde{C}$ the submanifold defined by $C\times (\{0\}\times\Omega)\subset\widetilde{X}$ (here we are regarding $\{0\}\times\Omega$ as a subset of $\widetilde{U}_0$ and $\widetilde{U}_1$). 
The second projection $\widetilde{U}_0\to \Omega$  (and the restriction $\widetilde{U}_1\to\Omega$ of this map) induces the submersion $\widetilde{X}\to \Omega$. 

Clearly the fiber $X_\tau$ and the submanifold $\widetilde{C}\cap X_\tau\subset X_\tau$ satisfies the conditions $(1)$ and $(2)$ with the holonomy function $f(\xi)=\tau\cdot \xi+\xi^2$. 
Now we can easily check the conditions $(a), (b), (c), $ and $(d)$ by applying Corollary \ref{cor:main} (here we also used \cite[Theorem 6.6.5]{Be} and the fact that there exists uncountably many elements $\tau\in U(1)$ which satisfies the condition as in Corollary \ref{cor:main} \cite[p. 155]{C}, see also \cite[\S 5.4]{U83}). 
\end{example}


\end{document}